\documentclass{elsarticlenoj}

\usepackage{amssymb}
\usepackage{amsmath}

\usepackage{amsthm}
\usepackage{latexsym}

\usepackage{rotating}

\newtheorem{theorem}{Theorem}[section]
\newtheorem{lemma}[theorem]{Lemma}
\newtheorem{corollary}[theorem]{Corollary}
\newtheorem{proposition}[theorem]{Proposition}

\theoremstyle{definition}
\newtheorem{definition}[theorem]{Definition}
\newtheorem{example}[theorem]{Example}

\newtheorem{remark}[theorem]{\bf Remark}

\newcommand{\oomit}[1]{{}}

\newcommand{\bN}{\ensuremath{\mathbb{N}}}

\newcommand{\bC}{\ensuremath{\mathbb{C}}}

\newcommand{\cA}{\ensuremath{\mathcal{A}}}

\newcommand{\lbox}{\lower 1pt \hbox{$\,\ssq\,$} }

\renewcommand{\phi}{\ensuremath{\varphi}}
\renewcommand{\epsilon}{\ensuremath{\varepsilon}}

\renewcommand{\leq}{\ensuremath{\leqslant}}
\renewcommand{\geq}{\ensuremath{\geqslant}}

\newcommand{\ssq}{\raise 1pt\hbox{$\scriptscriptstyle \square$}}

\newcommand \dstar{\mbox{\rm (\kern-0pt $\genfrac{}{}{0pt}{}{\raisebox{0pt}{$\star$}}{\raisebox{1pt}{$\star$}}$\rm)}}

\newcounter{smallromans}

\newenvironment{romanenumerate}
{\begin{list}{{\normalfont\textrm{(\roman{smallromans})}}}
  {\usecounter{smallromans}\setlength{\itemindent}{0cm}
   \setlength{\leftmargin}{5ex}\setlength{\labelwidth}{4ex}
   \setlength{\topsep}{0.75\parsep}\setlength{\partopsep}{0ex}
   \setlength{\itemsep}{0ex}}}
{\end{list}}

\newcommand{\tensor}{\widehat{\otimes}}

\newcommand{\Acal}{\mathcal{A}}

\newcommand{\BAI}{bounded approx\-imate identity}

\newcommand{\AM}{amen\-able}
\newcommand{\AMy}{amen\-ability}

 \newcommand{\BAA}{boundedly approx\-imately amen\-able}
 \newcommand{\BAAy}{boundedly approx\-imate amen\-ability}
\newcommand{\ApA}{approx\-imately amen\-able}

\newcommand{\ApAy}{approx\-imate amen\-ability}
 \newcommand{\BApA}{boundedly approx\-imately amen\-able}
 \newcommand{\BAC}{boundedly approx\-imately contract\-ible}
 \newcommand{\AC}{approx\-imately contract\-ible}

\begin{document}



\author{F. Ghahramani\fnref{labelG}}
\address{Department of Mathematics 
University of Manitoba, Winnipeg, R3T 2N2, Canada}
\ead{Fereidoun.Ghahramani@umanitoba.ca}

\author{R. J. Loy\corref{cor2}}
\address{Mathematical Sciences Institute, 
Building 27, Union Lane, Australian National University, 
Canberra, ACT 2601, Australia}
\ead{Rick.Loy@anu.edu.au}

\title
{Approximate amenability of tensor products of Banach  algebras} 

\fntext[labelG]{Supported by NSERC grant 36640-2012, and MSRVP at ANU.}



\newcommand\timestring{\begingroup
\count0=\time \divide\count0 by 60
\count2 = \count0
\count4 = \time \multiply\count0 by 60
\advance\count4 by -\count0
\ifnum\count4 <10 \toks1={0}%
\else \toks1 = {}%
\fi
\ifnum\count2<12 \toks0 = {a.m.}%
\else \toks0 = {p.m.}%
\advance\count2 by -12
\fi
\ifnum\count2=0 \count2 = 12
\fi
\number\count2:\the\toks1 \number\count4
\thinspace \the\toks0
\endgroup}

\newcommand\timenow{ \timestring \enskip
{\ifcase\month\or January\or February\or March\or 
  April\or May\or June\or July\or August\or Sepember\or
  October\or November\or December\fi}
\number\day}

\begin{abstract} 
Examples constructed by the first author and Charles Read make it  clear that many of the hereditary properties of amenability no longer hold  for approximate amenability.  These and earlier results of the authors also show that the presence of a \BAI\ often facilitates positive results.  
Here we show that the tensor product of \ApA\ algebras need not be \ApA, and investigate conditions under which
   $A$ and $B$ being \ApA\ implies, or is implied by, $A\tensor B$ or $A^{\#}\tensor B^{\#}$ being \ApA. Once again, the role of having a \BAI\ comes to the fore.
   Our methods also enable us to prove that if $A\tensor B$ is amenable, then so are $A$ and $B$, a result proved by Barry Johnson in 1996 under an additional assumption. 
\end{abstract}


\maketitle 

\vskip2mm 
\emph{In memoriam Charles John Read -- mathematician, gentleman and friend}

\section{Introduction}  \label{Introduction}
The concept of amenability for a Banach algebra, introduced by Johnson in 1972
\cite{John}, has proved to be of enormous importance in Banach algebra theory. 
In \cite{GhaLoy}, and subsequently in \cite{GhaLoyZh},  several modifications of this notion were introduced, in particular that of approximate amenability;  and  much work has been done in the last 10 years or so, \cite{DLZ, CGZ, CG, GhaSt, DL, GhaR1, GhaR2} for example. See also \cite{Zhang} for a recent survey.  In this paper the focus is on these newer notions for tensor products.  In particular, we investigate relations between  the approximate amenability of $A$ and $B$  and that of $A\tensor B$ or $A^{\#}\tensor B^{\#}$.

Let $A$ be an algebra, and let $X$ be an $A$-bimodule.  A \emph{derivation} is a linear map $D: A \to X$ such that
\[
D(ab)  = a\,\cdot \, D(b) + D(a)\,\cdot \,b \quad (a,b \in A)\,.
\]
For $x \in X$, set $ad_x: a \mapsto  a\,\cdot \, x - x\,\cdot \,a,\,\; A \to X$. Then $ad_x$ is a derivation; these are  the \emph{inner} derivations.

Let $A$ be a Banach algebra, and let $X$ be a
Banach $A$-bimodule.  A continuous derivation $D:A\to X$ is \emph{approximately inner} if there is
a net $(x_\alpha)$ in $X$ such that
\[
D(a) =
\lim_\alpha (a\,\cdot\, x_\alpha  - x_\alpha\,\cdot\, a)\quad (a \in A)\,,
\]
so that $D =\lim_\alpha ad_{x_\alpha}$ in the strong-operator topology  of ${\mathcal B}(A)$.

\begin{definition}\cite{GhaLoy}\label{app} Let $A$ be a Banach algebra. Then
$A$ is \emph{approximately amenable} (respectively \emph{approximately contractible})  if, for
each Banach $A$-bimodule $X$, every continuous derivation $D:A\to X^{*}$ (respectively $D:A\to X$), is approximately inner.  If it is always possible to choose the approximating net $(ad_{x_\alpha})$ to be bounded (with the bound dependent only on $D$) then $A$ is \emph{boundedly} \ApA\ (respectively, \emph{boundedly} \AC).
\end{definition}

Of course $A$ is \emph{amenable} (respectively, \emph{contractible}) if every continuous derivation $D:A\to X^{*}$ (respectively $D:A\to X$), is inner.

\vskip1mm
Of these various notions, amenability, contractibility, \ApA,  \BAA\ and \BAC\ are all distinct,  \AC\ and \ApA\ coincide, \cite{GhaLoyZh,  GhaR1, GhaR2}.
Also, requiring the approximating net of derivations to converge weak$^{*}$ is the same as \ApA, \cite{GhaLoyZh}. This latter notion will arise naturally below. 

\section{Some observations}

Recall the result of Barry Johnson \cite[Proposition 5.4]{John}:

\begin{proposition}\label{Barryprod}
Let $A$ and $B$ be amenable Banach algebras.  Then $A\tensor B$ is amenable.\hfill\qed
\end{proposition}

A version of this for the \ApA\ case was stated in \cite[Proposition 2.3]{GhaLoy}, but the argument there is incomplete. 
The matter is clarified in \cite[Proposition 4.1]{CG}, from which we state:
\begin{theorem}\label{GLcorr}
 Suppose that $A$ is a \BAA\ Banach algebra with a bounded
approximate identity, and that $B$ is an amenable Banach algebra. Then $A\tensor B$ is \BAA.\qed
\end{theorem}

In \cite{CG}  the question is raised whether the tensor product of two (boundedly) \ApA\ Banach algebras is itself (boundedly) \ApA.  We begin by answering this question in the negative.

\begin{theorem}\label{firstresult}
The tensor product of two \BAA\ Banach algebras need not be \ApA.
\end{theorem}

\begin{proof}
Let $A$ be the Banach algebra constructed in \cite{GhaR1} such that $A$ is \BAA\ yet $A\oplus A^{\textrm{op}}$ is not \ApA.
For convenience, set $B = A^{\textrm{op}}$.  Adjoin identities $1_A$ to $A$ and $1_B$ to $B$, and set $\Acal = A^{\#} \tensor B^{\#}$.  Then we have the decomposition into closed subspaces:
$$
\Acal = (\bC1_{A}\otimes 1_{B}) +(1_{A}\otimes B) + (A\otimes 1_{B}) + (A\tensor B)\,.
$$
Now $A\tensor B$ is a closed two -sided ideal in $\Acal$, so if $\Acal$ is \ApA,  the quotient algebra
$$
 (\bC1_{A}\otimes 1_{B}) +(1_{A}\otimes B) \oplus (A\otimes 1_{B})\simeq \Big((1_{A}\otimes B) \oplus (A\otimes 1_{B})\Big)^{\#
 }
 $$
 is also \ApA, whence so is $(1_{A}\otimes B) \oplus (A\otimes 1_{B})$ \cite[Proposition 2.4]{GhaLoy}.  The map
$$
( 1_A \otimes B) \oplus (A\otimes 1_B)\to B\oplus A :( 1_A \otimes b) + (a\otimes 1_B)\mapsto (b,a)
$$
is an isometric surjective algebra isomorphism, so  that  $A\oplus B$ is \ApA.  But this  contradicts the specific choice of $A$ and $B$.  Thus $\Acal$ cannot be  \ApA.
\end{proof}


Note that  the argument sheds no light on whether in this case the smaller $A\tensor B$ is \ApA.  

\begin{remark}The same example from \cite{GhaR1} also answers Question 1 raised in \cite[\S9]{GhaLoy}.  Namely $A\oplus B$ is not \ApA, yet the ideal $A$ is \BAA, as is the quotient $B$.
\end{remark}

We now build on this example to give a slightly sharper result. 
\begin{theorem}\label{countertensor}
There exists a unital \BApA\ Banach algebra $\cA$ such that $\cA\tensor \cA$ is not \ApA.
\end{theorem}

\begin{proof}
Let $A$ and $B$ be the \BApA\ algebras as above,  and set  $\cA = A^{\#}\oplus B^{\#}$, \BApA\ by \cite[Proposition 6.1]{GhaLoyZh}.  We have
$$
\Acal\tensor \Acal = (A^{\#}\tensor A^{\#}) \oplus (B^{\#}\tensor B^{\#}) \oplus (B^{\#}\tensor A^{\#}) \oplus(A^{\#}\tensor B^{\#}) \,,
$$
and 
$$
I = (A^{\#}\tensor A^{\#}) \oplus (B^{\#}\tensor B^{\#}) \oplus(B^{\#}\tensor A^{\#}) 
$$
\vskip1mm \noindent
 is a closed two-sided ideal.  The quotient $\Acal\tensor \Acal /I = A^{\#}\tensor B^{\#}$,  which is not \ApA\ from Theorem \ref{firstresult}, so that $\Acal\tensor \Acal $ is not \ApA.
\end{proof}

In comparison, note that since  \BAC\ algebras have bounded approximate identities \cite[Theorem 2.5]{CGZ},  so the direct sum of \BAC\ algebras is \BAC\ by a variant of \cite[Proposition 2.7]{GhaLoy}. 

There is a special situation when \ApAy\ of the tensor product comes for free.

\begin{proposition}\label{fa}
 Let $A$ and $B$ be Banach function algebras on their respective carrier spaces, and suppose that $A$ and $B$ have  bounded approximate identities consisting of elements of finite support.  Then $A\tensor B$ is \ApA.
\end{proposition}
\begin{proof} That $A$ and $B$ are \ApA\ follows from \cite[Proposition 4.2]{GhaSt}.   Now $A$ and $B$ have the bounded approximation property, so by \cite[\S3.2.18]{Pal} $A\tensor B$ is semisimple, and so is again a Banach function algebra.  It also has a \BAI\ of elements of finite support, built from those of $A$ and $B$, and once more \cite[Proposition 4.2]{GhaSt} applies.
\end{proof}

The same style of argument as above using  compositions can also give some positive results. 

\begin{theorem}  \label{timesplus}
Suppose that $A^{\#}\tensor B^{\#}$ is \ApA.  Then $A$, $B$ and $A\oplus B$ are \ApA.
\end{theorem}

\begin{proof}
The algebra $A^{\#}$ admits a character $\varphi$, and $a\tensor b\mapsto \varphi(a)b$ defines an epimorphism  $A^{\#}\tensor B^{\#}\to B^{\#}$ so that $B^{\#}$, and hence $B$, is \ApA.  Similarly for $A$. 

We have the decomposition into closed subalgebras,
$$
A^{\#}\tensor B^{\#}= (\bC 1_A \otimes  1_B) + (1_A \otimes B) + (A\otimes 1_B) + (A\tensor B)\,.
$$
Here $A\tensor B$ is a closed ideal, with \ApA\ quotient 
$$
 (\bC 1_A \otimes  1_B) + (1_A \otimes A)\oplus (B\otimes 1_B)
 $$
 having zero product between the second and third summands. But this latter is isomorphic to the unitization of $A\oplus B$.
 
\vskip-\baselineskip
 \end{proof}
 
The obvious omission here is whether $A\tensor B$ is \ApA.   This is certainly the case under an additional hypothesis.


\begin{theorem}  \label{sharp}
Suppose that $A^{\#} \tensor B^{\#}$ is (boundedly) \ApA\ and that $A$ and $B$ have bounded approximate identities.  Then $A\tensor B$ is (boundedly)  \ApA.
\end{theorem}

\begin{proof}
Since $A\tensor B$ has a bounded approximate identity, it suffices to show that for every essential Banach $A\tensor B$-bimodule $X$, continuous derivations from $A\tensor B$ into $X^*$ are approximately inner.

Let $D: A\tensor B\to X^{*}$ be a continuous derivation.  Then $D$ extends uniquely to a derivation $\widehat{D} : \Delta(A\tensor B)\to X^*$,  where $ \Delta(A\tensor B)$ is the double centralizer algebra of $A\tensor B$,  \cite{John1, John}.  Then restrict $\widehat{D}$ to $A^{\#} \tensor B^{\#}$.  By hypothesis this restriction is approximately inner, a fortiori, so is $D$. 

\end{proof}

\begin{remark}

An alternate proof would be to use the argument of \cite[Corollary 2.3]{GhaLoy}.   \end{remark}

\begin{remark}

1.  A possibly related question is whether $c_0(A)$ is \ApA\ given that $A$ is \ApA. The argument of \cite[Example 6.1]{GhaLoy} shows that $c_0(A^{\#})$ will be \ApA.  For the algebra $\cA$ of \cite{GhaR1}, $c_0(\cA)$ is again of the specified form of \cite[Theorem 3.1]{GhaR1}, and so is  \ApA. The more general question as to whether  $c_0(A_n)$ is \ApA, where the $(A_n)$ are \ApA, has a negative answer in general, as shown by the example $\cA\oplus\cA^{\textrm{op}}$ of \cite{GhaR1}.   

2.   Note that $c_0\tensor A$ will be \ApA\ if $A$ is \BAA\ and has a \BAI\ (Proposition \ref{GLcorr}).  For more general $A$ the question is open.  Of course there is a natural homomorphism $c_0\tensor A \to c_0(A)$ determined by $ (\alpha_n)\otimes x \mapsto (\alpha_n x)$.  Since elements of the range are summable sequences of elements of  $A$, the homomorphism has properly dense range.  Supposing that $c_0 \tensor A$ is \ApA\ it is not known whether $c_0(A)$ must be \ApA. However the epimorphism $c_0\tensor A \to A\oplus A$ determined by $ (\alpha_n)\otimes x \mapsto (\alpha_1 x, \alpha_2 x)$ shows that $A\oplus A$ would be.

\end{remark}

\section{Semi-inner derivations}\label{new method}
We first introduce a new notion which will come up in later arguments of this section.  The concept itself is not new, but the variant we wish to use seems to be.
\begin{definition}
Let $A$ be a algebra, $X$ an $A$-bimodule.  A derivation $D:A\to X$ is \emph{semi-inner}\footnote{Such maps, without the derivation condition, are called generalized inner, or  `generalized inner derivations' in the  literature \cite{AM, Br, CM}.  We require the approximate version, and view `approximately generalized' as an oxymoron, and so will use `semi-inner', but only for derivations.} if there are $m, n\in X$ such that 
$$
D(a) = a\cdot m - n\cdot a \qquad (a\in A)\,.
$$
More generally, for $A$ a Banach algebra, $X$ a Banach $A$-bimodule, a continuous derivation  $D:A\to X$ is \emph{approximately semi-inner} if there are nets $(m_i), (n_i)$ in $X$ with
$$
D(a) = \lim_i(a\cdot m_i - n_i\cdot a) \qquad (a\in A)\,.
$$
\end{definition}
 
Here $m$ and $n$ may be distinct but are highly constrained.  The derivation identity  shows that  if $D$ is a  semi-inner  derivation then
  $$
 a\cdot (m - n)\cdot b = 0\qquad (a, b\in A)\,.
 $$
Thus in the Banach case, with $D :A\to X^*$, then $m=n$ if $X$ is neo-unital, and we have an inner derivation.  In the approximately semi-inner case
$$
\lim_i \big(a\cdot (m_i - n_i)\cdot b\big) = 0 \qquad (a, b\in A)\,,
$$
and for $X$ neo-unital this latter gives $m_i - n_i\to 0$ weak$^*$, so that $D$ is in fact weak$^*$ approximately inner.
\begin{example}
The algebra $\ell^2$ under pointwise operations is not \ApA, \cite{DLZ, CG}.  It is, however, approximately semi-amenable.
For let $D:\ell^2 \to X$ be a continuous derivation into a  bimodule. 
Set $(E_n)$ to be the standard (unbounded) approximate identity of $\ell^2$.  Then $D_n : E_n\ell^2\to X$ is a derivation from a finite-dimensional semisimple algebra and hence is inner, say implemented by $\xi_n\in X$.  Thus for $a\in \ell^2$,
\begin{eqnarray*}
&&D(a) = \lim_n D(E_n a) = \lim_n (E_n a\cdot \xi_n -\xi_n\cdot E_n a)\\
&&\hskip 0.868cm = \lim\big( a\cdot (E_n\cdot \xi_n) - (\xi_n\cdot E_n)\cdot a\big)\,,
\end{eqnarray*}
and so $D$ is approximately semi-inner.
 \end{example}

\begin{theorem}\label{basic method}
Suppose that $A\tensor B$ is one of
\begin{romanenumerate}
\item \ApA\,,
\item \BAA\,,
\item \BAC\,.
\end{romanenumerate}
Then for any continuous derivation $D$ from $A$ or $B$ 
\begin{itemize}
\item into any  bimodule is approximately semi-inner in case (i)\,,
\item into any dual bimodule is boundedly  approximately semi-inner in case (ii)\,,
\item into any bimodule is boundedly approximately semi-inner in case (iii)\,.
\end{itemize}
\end{theorem}

\begin{proof}

Given a Banach $A$-bimodule $X$, we make $X\tensor B$ into a Banach $A\tensor B$-bimodule as follows:  for $a\in A, b_1\in B, b_2\in B, x \in X$,
$$
(a\otimes b_1)\cdot (x \otimes b_2) = a\cdot x\otimes b_1b_2\,,\qquad 
(x\otimes b_2)\cdot (a \otimes b_1) = x\cdot a\otimes b_2b_1\,.
$$
Given a continuous derivation $D:A\to X$, we define $\Delta : A\tensor B\to  X\tensor B$ by setting
$$
\Delta(a\otimes b) = D(a)\otimes b\qquad (a\in A, b\in B)\,.
$$
Then
\begin{eqnarray*}
\Delta((a_1\otimes b_1)(a_2\otimes b_2)) &=& \Delta(a_1a_2\otimes b_1b_2)\\
&=&\big(D(a_1)\cdot a_2 + a_1\cdot D(a_2)\big)\otimes (b_1b_2)\\
&=&\big((D(a_1)\cdot a_2)\otimes b_1b_2\big) + \big((a_1\cdot D(a_2))\otimes b_1b_2\big)\\
&=&\big((D(a_1)\otimes b_1)\cdot (a_2\otimes b_2)\big) + \big((a_1\otimes b_1)\cdot (D(a_2)\otimes b_2\big)\,,
\end{eqnarray*}
so that $\Delta$ is a derivation.

In clause (i), since approximate amenability and approximate contractibility  coincide, \cite[Proposition 2.1]{GhaLoyZh},  there is a net $(m_i)$ in $X\tensor B$ such that for all $a\in A, b\in B$,
\begin{equation}\label{start}
\Delta(a\otimes b) = \lim_i \Big((a\otimes b)\cdot m_i - m_i\cdot (a\otimes b)\Big)\,.
\end{equation}
Let
$$
m_i = \sum_{k=1}^\infty x_{k, i} \otimes b_{k, i}\,,
$$
where $ x_{k, i}\in X, b_{k, i}\in B$.  Then
\begin{eqnarray}\label{exp10}
&&\kern -1cm D(a)\otimes b = \Delta(a\otimes b)\nonumber\\ 
&&= \lim_i\left(\sum_k (a\cdot x_{k, i}) \otimes bb_{k, i} -  \sum_k  (x_{k, i}\cdot a) \otimes b_{k, i}b\right)\,.
\end{eqnarray}

Fix $b_0\in B$ non-zero, and take $b_0^*\in B^*$ with $\langle b_0^*,b_0\rangle = 1$.  Applying the operator $T:X\tensor B \to X$ specified by  $T(x\otimes b)= \langle b_0^*, b\rangle x$ to both sides of \eqref{exp10} yields
  \begin{equation}\label{net10}
D(a)= \lim_i (a\cdot m'_i - n'_i\cdot a)\,,
 \end{equation}
where  $m'_i =\sum_k \langle b_0^*, b_0b_{k, i}\rangle x_{k, i}$, $n'_i =  \sum_k\langle b_0^*, b_{k, i}b_0\rangle x_{k, i}$.
\vskip2mm

 In clause (iii), the same argument with the extra condition that
 $$
 \|(a\otimes b)\cdot m_i - m_i\cdot (a\otimes b)\|\leq K\|a\| \|b\|\,.
 $$
yields
$$
\|a\cdot m'_i - n'_i\cdot a\|\leq K'\|a\|\,.
 $$

For clause (ii),  let $D:A\to X^*$ be a continuous derivation into a dual bimodule. Since $X^*\tensor B$ is unlikely to be a dual space, let alone a dual module,  view the derivation $\Delta$  as mapping into $(X^*\tensor B)^{**}$.  Then there is a net $(m_i)$ in $(X^*\tensor B)^{**}$ and a constant $K > 0$ such that for $a\in A, b\in B$,
\begin{equation}
D(a)\otimes b = \Delta(a\otimes b) = \lim_i \Big((a\otimes b)\cdot m_i - m_i\cdot (a\otimes b)\Big)\,,\qquad  \label{eq100}
\end{equation}
and
\begin{equation}
\|(a\otimes b)\cdot m_i - m_i\cdot (a\otimes b) \|\leq K \|a\|\ \|b\|\,.\qquad\qquad\ \ \label{eq200}
\end{equation}

Fix $b_0\in B$ of unit norm and take $b_0^*\in B^*$  with $b_0^*(b_0) = 1$. Let $S :X\to (X^*\tensor B)^*$ be specified by
$$
\langle S(x),  x^*\otimes b\rangle =  \langle x^*, x\rangle \langle b_0^*, b\rangle\,,\qquad (x\in X, b\in B)\,,
$$
and set $T = S^*:  (X^*\tensor B)^{**}\to X^*$.  Now take $m\in  (X^*\tensor B)^{**}, a\in A, b\in B$, and $x\in X$.  Then 
$$
\langle T((a\otimes b)\cdot m), x\rangle = \langle (a\otimes b)\cdot m, S(x)\rangle = \langle m,  S(x)\cdot (a\otimes b)\rangle\,.
$$
For $x^*\in X^*$ and $c\in B$,
\begin{align}
\langle S(x)\cdot (a\otimes b_0), x^*\otimes c\rangle &= \langle S(x), (a\otimes b_0)\cdot (x^*\otimes c)\rangle\qquad\\
 &=  \langle S(x), a\cdot x^* \otimes b_0c\rangle = \langle a\cdot x^*, x\rangle \langle b_0^*, b_0c\rangle\,.
\end{align}
Thus, setting  $m= \sum_k x_k^*\otimes b_k$, and  $x^{*}(m) =  \sum_k \langle b_0^*, b_0 b_k\rangle  x_k^* $\,,
$$
 T((a\otimes b_0)\cdot m) =  \sum_k \langle b_0^*, b_0 b_k\rangle a \cdot x_k^* = a\cdot x^{*}(m)\,,
 $$
where we have the estimate
$$
\|x^{*}(m)\| \leq \|b_{0}\|\,\|b_{0}^{*}\|\,\|m\|\,.
$$

 A general $m\in(X^*\tensor B)^{**}$ is the weak$^*$-limit of a net $(\mu_{\alpha})\subset X^{*}\tensor B$, bounded by $\|m\|$, and as an adjoint $T$ is weak$^*$ to weak$^*$-continuous.  It  follows that the associated net $(x^*(\mu_{\alpha}))$  is bounded and so has a  weak$^*$ limit point  $\xi^*\in X^*$ (depending on $m$) which satisfies
\begin{equation}\label{xi10}
T((a\otimes b_0) \cdot m) = a\cdot \xi^*\qquad (a\in A)\,.
\end{equation}
Similarly, there is $\eta^*\in X^*$ with
\begin{equation}\label{eta100}
T(m\cdot (a\otimes b_0)) = \eta^*\cdot a \qquad (a\in A)\,.
\end{equation}
Applying $T$ to \eqref{eq100} and \eqref{eq200} with $b = b_{0}$, gives nets $(m'_i)$ and $(n'_i)$ in $X^*$ with
\begin{align}
&& D(a) = \lim_i (a\cdot m'_i - n'_i\cdot a) \qquad (a\in A)\,,\qquad\quad\ \ \label{eq101}\\
&&\|a\cdot m'_i - n'_i\cdot a \|\leq K\|T\|\ \|a\| \qquad (a\in A)\,.\qquad\label{eq102}
\end{align}
\end{proof}

To get beyond semi-inner we first observe that if
\begin{equation}
D(a)= \lim_i (a\cdot m'_i - n'_i\cdot a)\qquad (a\in A)\,,
 \end{equation}
and $D$ is a continuous derivation, then
for $a_1, a_2 \in A$, 
\begin{eqnarray}\label{prod15}
&&D(a_1 a_2) = D(a_1)a_2 + a_1 D(a_2)\nonumber \\
&&\hskip2cm  = \lim_i\Big[(a_1\cdot m'_i - n'_i\cdot a_1) a_2 + a_1(a_2\cdot m'_i - n'_i\cdot a_2) \Big]
\end{eqnarray}
and 
\begin{equation}\label{prod25}
D(a_1 a_2) =  \lim_i(a_1a_2\cdot m'_i - n'_i\cdot a_1a_2)\,.
\end{equation}
Comparing \eqref{prod15} and \eqref{prod25} yields 
\begin{equation}\label{limits}
    \lim_i(a_1\cdot(m'_i - n'_i)\cdot a_2) = 0\,.    
\end{equation}
Moreover, in the ``bounded'' case, we have
\begin{equation}\label{bounds}
\|a_{1}\cdot (m'_{i} - n'_{i})\cdot a_{2}\|\leq 3K\|a_{1}\|\cdot \|a_{2}\|\,.
\end{equation}

We can now look at conditions that enable us show that $m'_i = n'_i$, or at least $m'_i-n'_i\to 0$.

\begin{theorem}\label{BEJ condition}
 Suppose that $A\tensor B$ is \ApA\ (respectively \BAA, \BAC).  If $B$ has an element $b_0$ with $b_0\not\in\{b_0b - bb_0 : b\in B\}\overline{\phantom{I}}$, then $A$ is \ApA\ (respectively \BAA,  \BAC).
  \end{theorem}
  
  \begin{proof}
  Choose the functional $b^{*}_{0}$ in the proof of Theorem \ref{basic method} to vanish on $\{b_0b - bb_0 : b\in B\}$.  Then the resulting nets $(m'_{i})$ and $(n'_{i})$  are the same.  Hence the result.
 \end{proof}

 Natural conditions on $B$ which are sufficient for the above condition are listed in \cite{John2}. Note that there is unfortunately no conclusion about \ApAy\ of $B$.  Of course in special situations more can be said. \\
 
 Throughout the next theorem $ G$ is a locally compact group and $L^{1}(G)$ is the usual group algebra of $G$. 
 
 \begin{theorem}\label{groupalg}
Suppose that $L^{1}(G)\tensor A$ is  \ApA\ (respectively  (boundedly) \ApA). Then $G$ is amenable and $A$ is  \ApA\ (respectively  boundedly \ApA).  Conversely, if $G$ is amenable and $A$ is boundedly \ApA\ with a \BAI, then  $L^{1}(G)\tensor A$ is boundedly \ApA.
\end{theorem} 

\begin{proof} Let $\Lambda:f\mapsto \int_{G} f$ be the augmentation character on $L^{1}(G)$.  Then $T: f\otimes a\mapsto \Lambda(f)a$ gives a continuous epimorphism of $L^{1}(G)\tensor A$ onto $A$.  Thus $A$ is  \ApA\ (respectively  boundedly \ApA).

Let $I_{0} =\textrm{Ker} \Lambda$.  Since $I_{0}\tensor A$ is a complemented ideal in $L^{1}(G)\tensor A$, by \cite[Corollary 2.4]{GhaLoy} it has a left approximate identity.  Hence $I_{0}$ has a left approximate identity \cite[Theorem 8.2]{DorW},  and so $G$ is amenable by  \cite[Theorem 5.2]{Willis}.

For the partial converse, $G$ amenable  implies $L^{1}(G)$ amenable, now apply Theorem \ref{GLcorr}.
\end{proof}

Note that if $\Lambda(f_{0})=1$ then $L^{1}(G)\to I_{0}: f\mapsto f - \Lambda(f)f_{0}$ is a bounded projection onto $I_{0}$,  whence it follows that the norm on $I_{0}\tensor A$ is equivalent to that inherited from $L^{1}(G)\tensor A$.  Hence the complementation fact.

 \begin{theorem}\label{identity}
 Suppose that $A\tensor B$ is \BAC\ (respectively  \BAA).   Suppose that one of $A$  or $B$ has an identity.  Then $A$ and $B$ are \BAC\ (respectively \BAA).
  \end{theorem}
  
  \begin{proof}
 Suppose that $B$ has an identity $e$. Then, by Theorem \ref{BEJ condition}, $A$ is \BAC\ (respectively \BAA).

Now let $X$ be a Banach $B$-bimodule.  Then 
$$
X = e\cdot X\cdot e + (1-e)\cdot X \cdot e + e\cdot X\cdot (1-e) + (1-e)\cdot X \cdot (1-e)
$$
 is a decomposition into submodules.  Given a continuous derivation $D:B\to X$, it decomposes into the sum of 4 derivations into $e\cdot X \cdot e$, $(1-e)\cdot X \cdot e$ etc.  The last three of these have trivial module action by $B$ on at least one side, so the corresponding derivations are inner.  Thus we may suppose that $e\cdot x =x = x\cdot e$ for $x\in X$. 
 
 Let $D:B\to X^{*}$ be a continuous derivation, and consider the nets given by Theorem \ref{basic method}.  For the  \BAC\ situation, use clause (iii), for \BAA\ use clause (ii).   Putting $a_{1} = a_{2}= e$ in \eqref{limits} we have $m_i - n_i\to 0$, so that \eqref{eq101} and \eqref{eq102} give $D$ is boundedly approximately inner. 
  \end{proof}
\oomit{
 \begin{lemma}\label{esstrick2}
Let $A$ have a \BAI. Suppose that any continuous derivation from $A$ into the dual of a neo-unital bimodule is weak$^*$-approximately inner.  Then $A$ is \ApA.
If, further, the implementing nets can themselves be chosen to be bounded, then $A$ is \AM.
\end{lemma}
\begin{proof}
Let $X$ be a general $A$-bimodule, $D:A\to X^*$ a continuous derivation. Let $(e_\alpha)$ be a \BAI\ for $A$.  By Cohen's factorization theorem, $X_{ess} = A\cdot X\cdot A$ is a neo-unital $A$-bimodule.  Let $E$ be a weak$^*$-limit point of the left multiplication operators on $X^*$ by the elements $e_\alpha$, $F$ similarly for right multiplication.  Then $E$ and $F$ are commuting projections on $X^*$, and give  a decomposition
\begin{equation}\label{decomp3}
X^* = EFX^* \oplus E(I-F)X^* \oplus (I-E)X^*\,.
\end{equation}
Correspondingly, set
$$
D_1 = EFD, D_2 = E(I-F)D, D_3 = (I-E)D\,.
$$
These are easily seen to be derivations into the corresponding summands in \eqref{decomp3}.  Now let $\phi\in (X_{ess})^*$, and extend it by Hahn-Banach to $\tilde{\phi}\in X^*$.
Then $\theta(\phi) = EF\tilde{\phi}$ is easily seen to be a well defined $A$-bimodule monomorphism of $ (X_{ess})^*$ into $EFX^*$.  It is surjective since for $x^*\in X^*$, $\theta(x^*|_{X_{ess}}) = EFx^*$.  Thus  $EFX^*$ is isomorphic to $(X_{ess})^*$, whence $D_1$ is weak$^*$-approximately inner.\marginpar{\bf careful} But note this weak$^{*}$-topology is that on $EFX^{*}$ implemented by $X_{ess}$. 
Now the actions of $A$ on the right of $E(I-F)X^*$ and on the left of $(I-E)X^*$ are trivial,  and since $A$ has a \BAI\ $D_{2}$ and D$_{3}$ are approximately
 inner.  It follows that $D$ is weak$^*$-approximately inner.  
Thus $A$ is weak$^*$-\ApA, and the result now follows from  \cite[Proposition 2.1]{GhaLoyZh}.
In the bounded case, the same steps yield a bounded net implementing the derivation.  Thus  $A$ is amenable  by the sufficient part of \cite[Proposition 1] {Gourd}.
\end{proof}
}

 \begin{lemma}\label{esstrick3}
Let $A$ have a \BAI. Suppose that any continuous derivation from $A$ into the dual of a neo-unital bimodule is boundedly weak$^*$-approximately inner.  Then $A$ is boundedly weak$^{*}$-approximately amenable, and so \ApA.  
\end{lemma}

\begin{proof}
Let $X$ be a general $A$-bimodule, $D:A\to X^*$ a continuous derivation. Let $(e_\alpha)$ be a \BAI\ for $A$.  By Cohen's factorization theorem, $X_{ess} = A\cdot X\cdot A$ is a neo-unital $A$-bimodule.  Let $E$ be a weak$^*$-limit point of the left multiplication operators on $X^*$ by the elements $e_\alpha$, $F$ similarly for right multiplication.  Then $E$ and $F$ are commuting projections on $X^*$, and give  a decomposition
\begin{equation}\label{decomp4}
X^* = EFX^* \oplus E(I-F)X^* \oplus (I-E)X^*\,.
\end{equation}
Correspondingly, set
$$
D_1 = EFD, D_2 = E(I-F)D, D_3 = (I-E)D\,.
$$
These are easily seen to be derivations into the corresponding summands in \eqref{decomp4}.  Now let $\phi\in (X_{ess})^*$, and extend it by Hahn-Banach to $\tilde{\phi}\in X^*$.
Then $\theta(\phi) = EF\tilde{\phi}$ is easily seen to be a well-defined $A$-bimodule monomorphism of $ (X_{ess})^*$ into $EFX^*$.  It is surjective since for $x^*\in X^*$, $\theta(x^*|_{X_{ess}}) = EFx^*$.  Thus  $EFX^*$ is isomorphic to $(X_{ess})^*$, whence $D_1$ is boundedly weak$^*$-approximately inner. Now this weak$^{*}$-topology is $\sigma((X_{ess})^{*}, X_{ess})$, which is clearly weaker than the restriction of $\sigma(X^{*}, X)$. The unit ball in $(X_{ess})^{*}$ is compact under both topologies by Banach-Alaoglu, and so the two topologies coincide on bounded sets in $(X_{ess})^{*}$.  Thus $D_{1}$ is boundedly weak$^*$-approximately inner considered as mapping into $X^{*}$.

Now the actions of $A$ on the right of $E(I-F)X^*$ and on the left of $(I-E)X^*$ are trivial,  and since $A$ has a \BAI, $D_{2}$ and D$_{3}$ are boundedly approximately  inner.  It follows that $D$ is boundedly weak$^*$-approximately inner.  

That $A$ is \ApA\  now follows from  \cite[Proposition 2.1]{GhaLoyZh}.


\end{proof}


\begin{remark}
1. The hypothesis here of the derivations being boundedly weak$^*$-approximately inner is used to get equality of two weak$^{*}$-topologies.  Subsequently, the boundedness is lost with the appeal  \cite[Proposition 2.1]{GhaLoyZh}.    It is not known whether  $A$ must be  \BAA. 

2. In  \cite[Proposition 2.1]{GhaLoyZh}, the argument  looses control over boundedness  as Goldstine is invoked on the implementing elements, which in general will be unbounded.  Indeed, since  \BAC\ gives a bounded approximate identity \cite[Corollary 3.4]{CGZ}, which \ApA\ algebras need not have \cite[Corollary 3.2]{GhaR1}, the implication (2) $\Rightarrow$ (1) of  \cite[Proposition 2.1]{GhaLoyZh} fails with the qualifier `bounded'.  It is not known whether (3) $\Rightarrow$ (2) fails.

 3.  Note that by Banach-Steinhaus sequentially weak$^{*}$-approximately inner implies  boundedly weak$^{*}$-approximately inner. 
  \end{remark}

 \begin{theorem}\label{AMBAI}
 Suppose that $A\tensor B$ is \BApA\ and that $A$ has a \BAI.  Then $A$ is \ApA.
  \end{theorem}
 \begin{proof}
 Let $D:A\to X^*$ be a continuous derivation into the dual of a neo-unital bimodule $X$. From Theorem \ref{basic method}(ii),  we have nets $(m'_i)$ and $ (n'_i)$ in $X^*$ such that
 \begin{equation}\label{net}
 D(a)= \lim_i (a\cdot m'_i - n'_i\cdot a)\qquad (a\in A)\,,
 \end{equation}
 and $\|a\cdot m'_i - n'_i\cdot a\|\leq K\|a\|$, 
 where from \eqref{limits} and \eqref{bounds}
 \begin{equation}\label{netdiff}
  \lim_i(a_1\cdot(m'_i - n'_i)\cdot a_2) = 0\,,\qquad \|a_{1}\cdot(m'_{1} - n'_{i})\cdot a_{2}\|\leq 3K\|a_{1}\|\cdot\|a_{2}\|\end{equation}
for $a_1, a_2 \in A$.

In particular, for a given $x\in X$, and $a_{1}, a_{2}\in A$,
$$
\langle m'_{i }- n'_{i}, a_{2}x a_{1}\rangle\to 0, \ \quad  |\langle m'_{i} - n'_{i}, a_{2}x a_{1}\rangle|\leq  3K\|a_{1}\|\cdot\|a_{2}\| \cdot\|x\|\,.
$$
Since $X$ is neo-unital, it follows that
$$
 \langle m'_i - n'_i , x \rangle\to 0\,,
$$
and letting  $a_{1}, a_{2}$ range over an  approximate identity with bound $M$,
$$
\| m'_{i} - n'_{i}\|\leq  3KM^{2}\,.
$$
Thus for $a\in A$,
\begin{equation}\label{appform}
D(a) = \textrm{weak}^*-\lim_i (a \cdot m'_i - m'_i\cdot a)\,,\qquad \|a\cdot m'_{i} - m'_{i}\cdot a \|\leq 4K\|a\|\,.
\end{equation}

So we have that derivations into duals of neo-unital bimodules are boundedly weak$^*$- approximately inner, and the result follows from Lemma \ref{esstrick3}.
 \end{proof}

  The unwanted `bounded' assumption of Lemma \ref{esstrick3} and Theorem \ref{AMBAI} can be removed at the expense of a stronger hypothesis on the \BAI.  However, with this assumption comes a bonus to the conclusion of Theorem \ref{AMBAI}.

 \begin{theorem}\label{AMCBAI}
 Suppose that $A\tensor B$ is \ApA\ and that one of $A$ or $B$ has a central \BAI.  Then $A$ and $B$ are  \ApA.
 \end{theorem}
 \begin{proof}  Suppose that $(e_{\alpha})$ is a central \BAI\ in $B$.  Let $D:B\to X^*$ be a continuous derivation into the dual of a  bimodule $X$.  From Theorem \ref{basic method}(i),  we have a nets $(m'_i)$ and $ (n'_i)$ in $X^*$ such that
 \begin{equation}\label{nett}
 D(b)= \lim_i (b\cdot m'_i - n'_i\cdot b)\qquad (b\in B)\,,
 \end{equation}
 and
 \begin{equation}\label{diff}
  \lim_i(b_1\cdot(m'_i - n'_i)\cdot b_2) = 0\qquad (b_1, b_2 \in B)\,.
\end{equation}

Now follow Lemma \ref{esstrick3} to get $D_{1},  D_{2}$ and $D_{3}$. 
Then for $b\in B$,
\begin{eqnarray}
D_{1}(b) &=& (w^{*}-\lim_{\alpha})(w^{*} -\lim_{\beta})e_{\alpha}D(b)e_{\beta}\nonumber\\
&=&
(w^{*}-\lim_{\alpha})(w^{*} -\lim_{\beta})\lim_{i}[e_{\alpha}(b\cdot m'_{i} - n'_{i}\cdot b)e_{\beta}]\,. \label{Dlimit}
\end{eqnarray}
 
 Then \eqref{diff} and \eqref{Dlimit} give, using centrality of the \BAI,
 $$
 D_{1}(b)  =(w^{*}-\lim_{\alpha})(w^{*} -\lim_{\beta})\lim_{i}[b\cdot(e_{\alpha}\cdot m'_{i}\cdot e_{\beta}) - (e_{\alpha}\cdot n'_{i}\cdot e_{\beta})\cdot b]\,. 
$$
 Thus the standard method of considering finite subsets of $B$ and $X$, gives a net $(x^{*}_{\gamma})\subset X^{*}$ such that
 $$
 D_{1}(b) = w^{*} -\lim_{\gamma} (b\cdot x^{*}_{\gamma} - x^{*}_{\gamma}\cdot b)\,, \qquad (b\in B)\,.
 $$
 Since $D_{2}$ and $D_{3}$ are approximately inner we finally deduce that $D$ is weak$^{*}$-approximately inner. Thus $B$ is \ApA.

That $A$ is \ApA\ is now a consequence of Theorem \ref{BEJ condition}.
 \end{proof}

 Finally, an application of our method that improves  on the result \cite[Proposition 3.5]{John2}.
 
 \begin{theorem}\label{amenable case}
 Suppose that $A\tensor B$ is amenable.  Then $A$ and $B$ are amenable.
\end{theorem}
\begin{proof}  
Amenability of $A\tensor B$  implies it has a \BAI, whence so do $A$ and $B$, \cite[Theorem 8.2]{DorW}.  Now arguing as in Theorem \ref{basic method}  until at \eqref{start} and using  the necessary part of \cite[Proposition 1]{Gourd} we obtain a bounded net $(m_{i})$.  Then proceed as before to \eqref{appform} where the net $(m'_i)$ is now bounded.  Now use second part of Lemma \ref{esstrick3} to see that derivations from $A$ into a dual module are weak$^{*}$-approximately inner, with a bounded net of implementing elements.  The argument of \cite[Proposition 2.1]{GhaLoyZh} now shows that any continuous derivation into any $A$-bimodule  is approximately inner with a bounded net of implementing elements, that is, $A$ is amenable  by the sufficient part of \cite[Proposition 1] {Gourd}.
 \end{proof}

\section*{Acknowledgement}
The authors wish to thank Yong Zhang for pointing out some gaps in an earlier version of the paper.


\begin{thebibliography}{88}

\bibliographystyle{elsarticle-num}


\bibitem{Br} M.~Bresar, On the distance of the composition of two derivations to the generalized derivations, \emph{ Glasgow Math. J.} 33 (1991), 89--93. 

\bibitem{AM} P.~Ara and M.~Matheiu, \emph{Local multipliers of $C^{*}$-algebras}, Springer Monographs in Mathematics, Springer-Verlag, London 2003. 

\bibitem{CG} Y.~Choi and F.~Ghahramani, Approximate amenability of Schatten classes, Lipschitz algebras, and second duals of Fourier algebras, \emph{ Quart. J. Math.} 62 (2011), 39--58. 

\bibitem{CGZ} Y.~Choi, F.~Ghahramani and Y.~Zhang, Approximate and  pseudo-amenability of various classes of Banach algebras, \emph{J. Functional Analysis} 256(2009) 3158--3191. 

\bibitem{CM} R.~Curto and M.~Mathieu, Spectrally bounded generalized inner derivations, \emph{ Proc. Amer. Math. Soc.} 123 (1995), 2431--2434. 
\bibitem{DL} H.~G.~Dales and R.~J.~Loy, Approximate amenability of semigroup algebras and Segal algebras, \emph{Diss. Math.} 474(2010), 58 pages. 

\bibitem{DLZ} H.~G.~Dales, R.~J.~Loy, and Y.~Zhang, Approximate amenability for Banach sequence algebras, \emph{Studia Math.} 177(2006) 81--96. 

 
\bibitem{DorW}  R.~S.~Doran and  J.~Wichman, \emph{Approximate identities and factorization in Banach modules}, Lecture
Note in Mathematics, Vol. 768, Springer, New York, 1979. 

\bibitem{GhaLoy} F.~Ghahramani and R.~J.~Loy, Generalized notions of amenability, \emph{J. Functional Analysis} 208 (2004) 229--260.

\bibitem{GhaR1} F.~Ghahramani and C.~J.~Read, Approximate identities in approximate amenability, \emph{J. Functional Analysis} 262 (2012) 3929--3945. 

\bibitem{GhaR2} F.~Ghahramani and C.~J.~Read, Approximate amenability is not bounded approximate amenability, \emph{J. Math. Anal. Appl.} 423 (2015) 106--119. 

\bibitem{GhaSt} F.~Ghahramani and R.~Stokke, Approximate and pseudo-amenability of the Fourier algebra, \emph{Indiana Univ. Math. J.} 56 (2007) 909--930. 

\bibitem{GhaZh} F.~Ghahramani and Y.~Zhang, Pseudo-amenable and pseudo-contractible Banach algebras, \emph{Math. Proc. Camb. Phil. Soc.} 142 (2007) 111--123.

\bibitem{GhaLoyZh} F.~Ghahramani, R.~J.~Loy and Y.~Zhang, Generalized notions of amenability, II \emph{J. Functional Analysis} 254 (2008) 1776--1810. 

\bibitem{Gourd} F.~Gourdeau,  Amenability of Lipschitz algebra, \emph{Math. Proc. Camb. Phil. Soc.} 112 (1992) 581--588. 

\bibitem{John1} B.~E..~Johnson, An introduction to the theory of centralizers, \emph{Proc. London Math. Soc.} 14 (1964) 299--320.

\bibitem{John} B.~E.~Johnson, \emph{Cohomology in Banach algebras}, Mem. Amer. Math. Soc. 127 (1972). 

\bibitem{John2} B.~E.~Johnson, Symmetric amenability and the existence of Lie and Jordan derivations, \emph{Math. Proc. Camb. Phil. Soc.} 120 (1996) 455--473. 

\bibitem{Pal} T.~W.~Palmer, \emph{Banach algebras and the general theory of $*$-algebras. Volume 1, Algebras and Banach algebras}, Encylopedia of Mathematics and its Applications, Vol. 49, Cambridge University Press, Camnbridge, 1994.


\bibitem{Willis} G.~A.~Willis, Approximate units in finite codimensional ideals of group algebras, \emph{J. London Math. Soc.} (2), 26 (1982) 143--154. 

\bibitem{Zhang} Y.~Zhang, Solved and unsolved problems in generalized notions of amenability for Banach algebras, in \emph{Banach algebras 2009},
Banach Center Publ., 91, Polish Acad. Sci. Inst. Math. Warsaw, 2010. 

\end{thebibliography}
\end{document}